\documentclass[reqno, 11pt]{amsart}

\usepackage[english]{babel}
\usepackage{amsmath}
\usepackage{amssymb}
\usepackage{enumerate}
\usepackage{ifthen}
\usepackage{bbm}
\usepackage{color}
\provideboolean{shownotes} 
\setboolean{shownotes}{true}
\usepackage{hyperref}
\newtheorem{theorem}{Theorem}[section]

\newtheorem{lemma}[theorem]{Lemma}

\newtheorem{definition}[theorem]{Definition}

\theoremstyle{remark}
\newtheorem{remark}[theorem]{Remark}

\newcommand{\R}{\mathbb{R}}
\newcommand{\T}{\mathbb{T}^3}

\newcommand{\dive}{\mathop{\mathrm {div}}}

\newcommand{\pt}{\partial_t}
\newcommand{\rrho}{\sqrt{\rho}}
\newcommand{\la}{\lambda}
\newcommand{\be}{\beta\left(\frac{|w_{n}|^2}{2}\right)}
\newcommand{\bep}{\beta'\left(\frac{|w_{n}|^2}{2}\right)}
\newcommand{\bedp}{\beta''\left(\frac{|w_{n}|^2}{2}\right)}
\newcommand{\rn}{\rho_n}
\newcommand{\rrn}{\sqrt{\rho_n}}
\newcommand{\un}{u_{n}}
\newcommand{\wn}{w_{n}}

\def\d{{\partial}}
\DeclareMathOperator{\diver}{div}
\DeclareMathOperator{\symmD}{D}

\numberwithin{equation}{section}



\begin{document}

\title[Quantum Navier-Stokes]{On the compactness of finite energy weak solutions to the Quantum Navier-Stokes equations}

\author[P. Antonelli]{Paolo Antonelli}
\address[P. Antonelli]{\newline
GSSI - Gran Sasso Science Institute \\ Viale Francesco Crispi 7\\67100, L'Aquila \\Italy}
\email[]{\href{paolo.antonelli@gssi.infn.it}{paolo.antonelli@gssi.infn.it}}

\author[S. Spirito]{Stefano Spirito}
\address[S. Spirito]{\newline
GSSI - Gran Sasso Science Institute \\ Viale Francesco Crispi 7\\67100, L'Aquila \\Italy}
\email[]{\href{stefano.spirito@gssi.infn.it}{stefano.spirito@gssi.infn.it}}

\begin{abstract}
We consider the Quantum Navier-Stokes system in three space dimensions. We prove compactness of finite energy weak solutions for large initial data. 
The main novelties are that vacuum regions are included in the weak formulation and no extra terms, like damping or cold pressure, are considered in the equations in order to define the velocity field. Our argument uses an equivalent formulation of the system in terms of an effective velocity, in order to eliminate the third order terms in the new system. 
This will allow to obtain the same compactness properties as for the Navier-Stokes equations with degenerate viscosity.

\end{abstract}

\maketitle

\section{Introduction}
We study the quantum Navier-Stokes (QNS) system on $(0, T)\times\Omega$
\begin{align}
&\pt\rho+\dive(\rho u)=0,\label{eq:qns1}\\
&\pt(\rho u)+\dive(\rho u\otimes u)+\nabla\rho^{\gamma}-2\nu\dive(\rho Du)-2k^2\rho\nabla\left(\frac{\Delta\rrho}{\rrho}\right)=0,\label{eq:qns2}
\end{align}
with initial data
\begin{equation}\label{eq:id}
\begin{aligned}
\rho(0,x)&=\rho^0(x),\\
(\rho u)(0,x)&=\rho^0(x)u^0(x).
\end{aligned}
\end{equation} 
We deal with a three dimensional domain $\Omega$ which will be chosen to be either the whole space or a bounded domain with periodic boundary conditions, namely $\Omega=\R^3$ or $\Omega=\mathbb T^3$.
The unknowns $\rho, u$ represent the mass density and the velocity field, respectively, of the fluid, $\nu$ and $\kappa$ are positive constants and they are called the viscosity and the dispersive coefficients, respectively. 
In this paper we prove the compactness of bounded sets of finite energy weak solutions to \eqref{eq:qns1}-\eqref{eq:id}. 
Our result is new to our knowledge and is the first step towards the analysis of global in time finite energy weak solutions to the Cauchy problem \eqref{eq:qns1}-\eqref{eq:id}. The second step is the construction of a sequence of regular smooth solutions satisfying the a priori bounds stated in this paper. The analysis of the approximating sequence is very involved and it is performed in \cite{AS1}. Let us stress that the compactness result holds in a more general case than \cite{AS1}.  

Similar systems were extensively studied in the literature during the last years. In particular \eqref{eq:qns1}-\eqref{eq:qns2} belong to the more general class of the so called Navier-Stokes-Korteweg systems
\begin{align*}
&\d_t\rho+\diver(\rho u)=0\\
&\d_t(\rho u)+\diver(\rho u\otimes u)+\nabla p=\diver\mathbb S+\diver\mathbb K,
\end{align*}
where the viscosity stress tensor reads
\begin{equation*}
\mathbb S=\nu(\rho)\symmD u+\lambda(\rho)\diver u\mathbb I,
\end{equation*}
and the capillarity dispersive term is given by
\begin{equation*}
\diver\mathbb K=\nabla\left(\rho\diver(k(\rho)\nabla\rho)-\frac12(\rho k'(\rho)-k(\rho))|\nabla\rho|^2\right)-\diver(k(\rho)\nabla\rho\otimes\nabla\rho).
\end{equation*}
The dispersive tensor above has been derived in \cite{DS} to take into account capillarity phenomena in diffuse interfaces. For more details about capillary fluid and Navier-Stokes-Korteweg systems, we address the reader to the review \cite{BG} and references therein. It is straightforward to check that $\kappa(\rho)=\frac{\kappa}{4\rho}$, $\nu(\rho)=2\nu\rho$, $\lambda(\rho)=0$ give the QNS system \eqref{eq:qns1}-\eqref{eq:qns2}.

Furthermore, similar systems arise also in the description of quantum fluids. For example \eqref{eq:qns1}-\eqref{eq:qns2} with $\nu=0$ is the well known Quantum Hydrodynamics (QHD) model for superfluids \cite{LL}. Global existence of finite energy weak solutions for the QHD system has been studied in \cite{AM, AM2}. The viscous correction term in \eqref{eq:qns1}-\eqref{eq:qns2} has been derived in \cite{BM}, by closing the moments for a Wigner equation with a BGK term. For more details about the derivation of the QNS system we refer the reader to \cite{J2}.

In the study of the Cauchy problem \eqref{eq:qns1}-\eqref{eq:id} we need to overcome several mathematical difficulties. The former is that the viscosity coefficient is degenerate at vacuum. Even for the Navier-Stokes equations this case cannot be covered by the well-established theory by Lions and Feireisl, see \cite{F,L} and the recent approach in \cite{BJ}. 
Indeed, it is not possible to define the velocity field in the nodal region $\{\rho=0\}$. The compactness of solutions in this case is proved in \cite{MV} by using the Bresch-Desjardin entropy \cite{BD} and a logarithmic improvement on the integrability of $\rho|u|^2$, which is crucial to deal with the convective term in the vacuum region. Recently also the global existence of finite energy weak solutions has been obtained \cite{VY1, LX}. 
Furthermore, in the case of Quantum Navier-Stokes, the highly nonlinear third order dispersive term requires further regularity on the solutions to prove the compactness. 
Moreover, the difficulties about the convective term are the same as for the Navier-Stokes system with degenerate viscosity, but because of the dispersive term it is not possible to follow the approach in \cite{MV}. 
 
 J\"ungel in \cite{J} proves the global existence of finite energy weak solutions for \eqref{eq:qns1}-\eqref{eq:id} in the case when $\kappa>\nu$ and $\gamma>3$. By using an alternative formulation of the quantum Navier-Stokes system in terms of an effective velocity he is able to infer further a priori estimates on the second derivative for the mass density and to establish a maximum principle for the approximate solutions. However the notion of weak solutions in \cite{J} involves test functions of the type $\rho\phi$, with $\phi$ smooth and compactly supported. 
 This particular choice is due to the difficulty of defining the weak solutions in the nodal region, especially the convective term. This specific definition of weak solutions was introduced in \cite{BDL} to prove a global existence result for a Navier-Stokes-Korteweg system with a particular choice of viscosity and capillarity coefficients.
 
Considering the usual notion of weak solutions, some global existence results for \eqref{eq:qns1}-\eqref{eq:id} have been proved by augmenting the system with additional terms which allow to define the velocity field in the vacuum region. For example, in \cite{GLV} the authors consider a cold pressure term, which allows to estimate $1/\rho$ in some suitable Sobolev spaces, whereas in \cite{VY2} damping terms are added.

For the quantum Navier-Stokes system \eqref{eq:qns1}-\eqref{eq:id} it is possible to obtain a BD-type estimate, see for example \cite{J}. However a Mellet-Vasseur type estimate cannot be proven, due to the presence of the dispersive term. This is also remarked in \cite{VY1} where an approximate estimate of the type is proved, by exploiting the extra damping terms added to the system and a truncation technique for the mass density. 

In our paper we are going to prove a Mellet and Vasseur type estimate for weak solutions to \eqref{eq:qns1}-\eqref{eq:id} by using an alternative formulation of the system in terms of an effective velocity. Indeed, by choosing appropriately the constant in the effective velocity it is possible to tune the viscosity and capillarity coefficients. We choose it in such a way that the capillarity coefficient vanishes. This will yield the compactness for the effective velocity and consequently the same property will hold for weak solutions to \eqref{eq:qns1}-\eqref{eq:id}, by exploiting the BD entropy estimate.

Our paper is structured as follows.
In Section \ref{sect:defs}  we introduce the notations and definitions we are going to use in our paper, we then state our main result. Section \ref{sect:apriori} is devoted to the derivation of the a priori estimates which will then be used in Section \ref{sect:main} to show the convergence result of the main Theorem.
\section{Notations, Definitions and Main Result}\label{sect:defs}
In this section we are going to fix the notations used in the paper, to give the precise of definition of weak solution for the system \eqref{eq:qns1}-\eqref{eq:qns2} and to state our main result.\\

{\bf Notations}\\
\\
Given $\Omega\subset\R^3$, the space of compactly supported smooth functions will be $\mathcal{D}((0,T)\times\Omega)$. We will denote with $L^{p}(\Omega)$ the standard Lebesgue spaces and with $\|\cdot\|_p$ their norm. The Sobolev space of $L^{p}$ functions with $k$ distributional derivatives in $L^{p}$ is $W^{k,p}(\Omega)$ and in the case $p=2$ we will write $H^{k}(\Omega)$. The spaces $W^{-k,p}(\Omega)$ and $H^{-k}(\Omega)$ denote the dual spaces of $W^{k,p'}(\Omega)$ and $H^{k}(\Omega)$ where  $p'$ is the H\"older conjugate of $p$. Given a Banach space $X$ we use the the classical Bochner space for time dependent functions with value in $X$, namely $L^{p}(0,T;X)$, $W^{k,p}(0,T;X)$ and $W^{-k,p}(0,T;X)$. Finally, $Du=(\nabla u+(\nabla u)^T)/2$ is the symmetric part of the gradient.\\

{\bf Weak Solutions}\\
\\
Here and throughout the paper we will consider for simplicity the case $\Omega=\T$. The case $\Omega=\R^3$ can be considered with minor changes in the proofs.
We first recall two alternative ways to write the third order tensor term, which will be very useful in the sequel:
\begin{equation}\label{eq:quantum}
2\rho\nabla\left(\frac{\Delta\rrho}{\rrho}\right)=\dive(\rho\nabla^2\log\rho)=\nabla\Delta\rho-4\dive(\nabla\rrho\otimes\nabla\rrho).
\end{equation}
Then, by using \eqref{eq:quantum}, we can consider the following definition of weak solutions. 
\begin{definition}\label{def:ws}
A pair $(\rho, u)$ with $\rho\geq0$ is said to be a weak solution of the Cauchy problem \eqref{eq:qns1}-\eqref{eq:qns2}-\eqref{eq:id} if 
\begin{enumerate}
\item Integrability conditions:
\begin{align*}
&\rho\in L^{\infty}(0,T;L^{1}\cap L^{\gamma}(\T)),\\
&\rrho u\in L^{\infty}(0,T;L^{2}(\T)),\\
&\rrho\in L^{\infty}(0,T;H^{1}(\T)).\\
\end{align*}
\item Continuity equation:
\begin{equation*}
\int_\Omega \rho^0\phi(0)\,dx+\iint_0^T\int_\Omega\rho\phi_t+\rrho\rrho u\nabla\phi\,dxdt=0,
\end{equation*}
for any $\phi\in C_c^{\infty}([0,T);C^{\infty}(\T))$.\\
\item Momentum equation:
\begin{equation*}
\begin{aligned}
&\int_\T \rho^0u^0\phi(0)\,dx+\int_0^T\int_{\T}\rrho(\rrho u)\psi_t+\rrho u\otimes\rrho u\nabla\psi+\rho^{\gamma}\dive\psi\,dxdt\\
&-2\nu\int_0^T\int_{\T}(\rrho u\otimes\nabla\rrho)\nabla\psi\,dxdt-2\nu\int_0^T\int_{\T}(\nabla\rrho\otimes\rrho u)\nabla\psi\,dxdt\\
&-4\kappa^2\int_0^T\int_{\T}(\nabla\rrho\otimes\nabla\rrho)\nabla\psi\,dxdt+\kappa^2\int_0^T\int_{\T}\rho\Delta\nabla\phi\,dxdt\\
&+\nu\int_0^T\int_{\T}\rrho\rrho u\Delta\psi\,dxdt+\nu\int_0^T\int_{\T}\rrho\rrho u\nabla\dive\psi\,dxdt=0,
\end{aligned}
\end{equation*}
for any $\psi\in C_c^{\infty}([0,T);C^{\infty}(\T))$. 
\item Energy Inequality: if

\begin{equation*}
E(t)=\int\frac12\rho|u|^2+\frac{\rho^\gamma}{\gamma-1}+2\kappa^2|\nabla\sqrt{\rho}|^2\,dx,
\end{equation*}
then the following energy inequality is satisfied
\begin{equation*}
E(t)+2\nu\int_0^t\int_{\T}\rho|Du|^2\,dxdt'\leq E(0).
\end{equation*}
\end{enumerate}
\end{definition}
\bigskip
{\bf Main Result}\\
\\
Let us start by specifying the assumptions on the initial data. Let $\tau$ be a small fixed positive number. We consider  $\{\rho^0_n\}_{n}$ being a sequence of smooth and strictly positive functions and $\rho^0$ be a positive function such that 
\begin{equation}\label{eq:hyidr}
\begin{aligned}
&\rho_n^0>0,\quad \rho_n^0\rightarrow \rho\textrm{ strongly in }L^{1}(\T),\\
&\{\rho_n^0\}_{n}\textrm{ in uniformly in bounded in }L^{1}\cap L^{\gamma}(\T),\\
&\{\nabla\sqrt{\rho_n^0}\}_n\textrm{ is uniformly bounded in }L^{2}\cap L^{2+\tau}(\T),\\
&\nabla\sqrt{\rho_n^0}\rightarrow\nabla\sqrt{\rho^0}\textrm{ strongly in }L^{2}(\T).
\end{aligned}
\end{equation}
Regarding the initial velocity, let $\{u_0^{n}\}$ be a sequence of smooth vector fields such that 
\begin{equation}\label{eq:hyidu}
\begin{aligned}
&\{\sqrt{\rho_n^0}u_{n}^0\}\textrm{ is uniformly bounded in }L^{2}\cap L^{2+\tau}(\T),\\
&\rho_n^0 u_n^0\rightarrow \rho^0 u^0\textrm{ in }L^{1}(\R^3).\\ 
\end{aligned}
\end{equation}
The hypothesis of the higher integrability of $\nabla\sqrt{\rho_n^0}$ and $\sqrt{\rho_n^0}u_n^0$ imply that 
\begin{equation}\label{eq:mvi}
\rho_n^0\left(1+\frac{|w_n^0|^2}{2}\right)\log\left(1+\frac{|w_n^0|^2}{2}\right)\textrm{ is uniformly bounded in }L^{1}(\T),
\end{equation}
with $w_n^0=u_n^0+c\nabla\log\rho_n^0$ and $c>0$.\\

The main theorem of this paper is the following:
\begin{theorem}\label{teo:main}
Let $k<\nu$ and $1<\gamma<3$. Assume $\{\rn^0\}_n$ and $\{\rn^0\un^0\}_n$ are sequences of initial data for \eqref{eq:qns1}-\eqref{eq:qns2} satisfying \eqref{eq:hyidr} and \eqref{eq:hyidu}.\\
Let $\{(\rho_n, \un)\}_n$ be a sequence of smooth solutions of \eqref{eq:qns1}-\eqref{eq:qns2} with initial data $\{\rn^0\}_n$ and $\{\rn^0\un^0\}_n$ such that $\rn>0$. 
Then, up to subsequences not relabelled,  there exist $\rho$ and $u$ such that 
\begin{equation}\label{eq:main}
\begin{aligned}
&\nabla\rrn\rightarrow\nabla\rrho\textrm{ strongly in }L^{2}((0,T)\times\T),\\
&\rrn\un\rightarrow\rrho u\textrm{ strongly in }L^{2}((0,T)\times\T),\\
&\rn^{\gamma}\rightarrow\rho^{\gamma}\textrm{ strongly in }L^{1}((0,T)\times\T).
\end{aligned}
\end{equation} 
Moreover, $(\rho, u)$ is a finite energy weak solutions of \eqref{eq:qns1}-\eqref{eq:qns2}-\eqref{eq:id}.
\end{theorem}
\begin{remark}
Let us remark that the limit velocity field $u$ is not uniquely defined in the vacuum region $\{\rho=0\}$.
\end{remark}
\begin{remark}
The smoothness assumption of $(\rho_n,\un)$ does not play any role in the compactness result and it is made in order have a more neat statement. Indeed, the Theorem holds true also by assuming 
$\{(\rn, \un)\}_n$ is a sequence of weak solutions in the sense of Definition \ref{def:ws} satisfying the uniform bounds obtained from Lemma \ref{lem:s1}, Lemma \ref{lem:s3}, Lemma \ref{lem:s3bis} and Lemma \ref{lem:s4} and the assumptions on the initial data \eqref{eq:hyidr} and \eqref{eq:hyidu}.
\end{remark}
\begin{remark}
We stress that while Theorem \ref{teo:main} holds in the case $\kappa<\nu$, the existence result in \cite{AS1} requires that  
$\alpha\nu<\kappa<\nu$ for some $\alpha<1$. However that restriction is required only in the three dimensional case. 
\end{remark}

\section{A priori estimates}\label{sect:apriori}
In this section we prove the a priori estimates needed in the proof of the convergence. The first lemma is the basic energy inequality associated to the (QNS) system. 
\begin{lemma}\label{lem:s1}
Let $(\rn,\un)$ be a smooth solution of \eqref{eq:qns1}-\eqref{eq:qns2}, then 
\begin{equation}\label{eq:en}
\frac{d}{dt}\left(\int\rho\frac{|\un|^2}{2}+\frac{\rn^{\gamma}}{\gamma-1}+2{\kappa}^2|\nabla\rrn|^2\,dx\right)+2\nu\int\rn|D\un|^2\,dx=0.
\end{equation}
\end{lemma}
\begin{proof}
The proof is straightforward and can be found in \cite{J}.
\end{proof}
As we already said in the Introduction we are now going to consider an effective velocity $w:=u+c\nabla\log\rho$. 
Then \eqref{eq:qns1}-\eqref{eq:qns2} can be transformed in a modified system in terms of the variables $(\rho, w)$, where the viscosity and capillarity coefficients depend on $c$. We can thus choose that constant such that the dispersive term vanishes.
This will allow us to infer the desired a priori estimates and consequently the necessary compactness properties for the weak solutions. A similar approach, i.e. the introduction of an effective velocity for which the modified system has more convenient a priori estimates, was already used in \cite{J} (see also \cite{BD}) to prove global existence of weak solutions to \eqref{eq:qns1}-\eqref{eq:qns2} in the case $\kappa>\nu$, $\gamma>3$.
\newline
Let us also remark that $c\nabla\log\rho$ has the dimensions of a velocity; this is indeed the so-called \emph{osmotic velocity} appearing in the context of Nelson formulation of quantum mechanics, see for instance \cite{Car} and references therein.
\newline
Furthermore, in \cite{BGDD} the complex vector field $w=u+ic\nabla\log\rho$ has been used to reformulate the Euler-Korteweg system as a quasilinear Schr\"odinger equation.
\begin{lemma}\label{lem:s3}
Let $(\rn,\un)$ be a smooth solution of \eqref{eq:qns1}-\eqref{eq:qns2}. Given $\nu>\kappa$ let $\mu=\nu-\sqrt{\nu^2-{\kappa}^{2}}$, $\la=\sqrt{\nu^2-{\kappa}^2}$ and $\wn=\un+\mu\nabla\log\rn$. Then, $(\rn,\wn)$ satisfies 
\begin{align}
&\pt\rn+\dive(\rn\wn)=\mu\Delta\rn, \label{eq:wqns1}\\
&\pt(\rn \wn)+\dive(\rn \wn\otimes \wn)+\nabla\rn^{\gamma}-\mu\Delta(\rn \wn)-2\la\dive(\rn D\wn)=0\label{eq:wqns2}.
\end{align}
\end{lemma}
\begin{proof}
Let $c\in \R$ to be chosen later. Let us consider $\wn=\un+c\nabla\log\rn$. Then,
\begin{equation*}
\pt\rn+\dive(\rn\wn)=\pt\rn+\dive(\rn(\un+c\nabla\log\rn))=c\Delta\rn.
\end{equation*}
To prove equation \eqref{eq:wqns2} we first recall the following elementary identities,
\begin{align*}
c(\rn\nabla\log\rn)_t&=-c\nabla\dive(\rn \un),\\
c\dive(\rn \un\otimes\nabla\log\rn+\rn\nabla\log\rn\otimes \un)&=c\Delta(\rn \un)-2c\dive(\rn D\un)\\&+c\nabla\dive(\rn \un),\\
c^{2}\dive(\rn\nabla\log\rn\otimes\nabla\log\rn)&=c^2\Delta(\rn\nabla\log\rn)\\
&-c^2\dive(\rn\nabla^2\log\rn).
\end{align*}
Moreover, by using \eqref{eq:quantum}
\begin{equation*}
2\rn\nabla\left(\frac{\Delta\rrn}{\rrn}\right)=\dive(\rn\nabla^2\log\rn).
\end{equation*}
By using these identities it is easy to prove that 
\begin{equation*}
\begin{aligned}
\pt(\rn \wn)+\dive(\rn \wn\otimes \wn)+\nabla\rn^{\gamma}&-c\Delta(\rn \wn)=2(\nu\!-\!c)\dive(\rn D\wn)\\
&+(k^2\!-\!c^2\!\!-\!2(\nu\!-\!c)c)\dive(\rn\nabla^2\log\rn).
\end{aligned}
\end{equation*}
Then, by imposing that 
\begin{equation*}
\begin{aligned}
k^2-c^2-2(\nu-c)c&=0,\\
\nu-c>&\,0,
\end{aligned}
\end{equation*}
we get that $c=\mu=\nu^2-\sqrt{\nu^2-k^2}$ and $\nu-c=\la=\sqrt{v^2-k^2}$. 
\end{proof}
Once we have written \eqref{eq:qns1}-\eqref{eq:qns2} in terms of the variable $\wn$ we can perform further a priori estimates.
\begin{lemma}\label{lem:s3bis}
Let $(\rn,\un)$ be a smooth solution of \eqref{eq:qns1}-\eqref{eq:qns2}. Then, $\wn=\un+\mu\nabla\rn$ and $\rn$ satisfy
\begin{multline}\label{eq:bd1}
\frac{d}{dt}\left(\int\rn\frac{|\wn|^2}{2}+\frac{\rn^\gamma}{\gamma-1}\,dx\right)+\\
+\frac{4\mu}{\gamma}\int|\nabla\rn^{\frac{\gamma}{2}}|^2\,dx
+2\la\int\rn| D \wn|^2\,dx+\mu\int\rn|\nabla \wn|^2\,dx=0.
\end{multline}
\end{lemma}
\begin{proof}
Estimate \eqref{eq:bd1} is nothing but the energy estimate associated to system \eqref{eq:wqns1}-\eqref{eq:wqns2}.
By multiplying \eqref{eq:wqns2} by $w_n$, by integrating in space and by using \eqref{eq:wqns1} we get 
\begin{equation}\label{eq:s4}
\frac{d}{dt}\int\rn\frac{|\wn|^2}{2}\,dx+\int\nabla\rn^{\gamma} \wn\,dx+2\la\int\rn|D \wn|^2\,dx+\mu\int\rn|\nabla \wn|^2\,dx=0.
\end{equation}
Then, we multiply the \eqref{eq:wqns1} by $\gamma\rn^{\gamma-1}/(\gamma-1)$ and by integrating by parts we get 
\begin{equation}\label{eq:s5}
\frac{d}{dt}\int\frac{\rn^{\gamma}}{\gamma-1}\,dx-\int\nabla\rn^{\gamma} \wn\,dx-\mu\gamma\int\Delta\rn\frac{\rn^{\gamma-1}}{\gamma-1}\,dx=0.
\end{equation}
By integrating by part the last term and by summing up \eqref{eq:s4} and \eqref{eq:s5} we get \eqref{eq:bd1}.
\end{proof}
The next a priori estimate is the key tool of the compactness analysis. It was first discovered by Mellet and Vasseur in \cite{MV}. 
\begin{lemma}\label{lem:s4}
Let $(\rn,\un)$ be a smooth solution of \eqref{eq:qns1}-\eqref{eq:qns2}. Then, $\wn=\un+\mu\nabla\rn$ and $\rn$ satisfy
\begin{equation}\label{eq:mv}
\begin{aligned}
&\frac{d}{dt}\int\rn\left(1+\frac{|\wn|^2}{2}\right)\log\left(1+\frac{|\wn|^2}{2}\right)\,dx+\mu\int\rn|\nabla \wn|^2\,dx\\
&+\mu\int\rn|\nabla \wn|^2\log\left(1+\frac{|\wn|^2}{2}\right)\,dx+2\la\int\rn|D\wn|^2\,dx\\
&+2\la\int\rn|D\wn|^2\log\left(1+\frac{|\wn|^2}{2}\right)\,dx+\mu\int\rn\left|\nabla\frac{|\wn|^2}{2}\right|^2\frac{2}{2+|\wn|^2}\,dx\\
&=-2\la\int\rn D\wn \wn\nabla \wn \wn\frac{2}{2+|\wn|^2}\,dx\\
&-\int\rn^\gamma\dive\wn\left(1+\log\left(1+\frac{|\wn|^2}{2}\right)\right)\,dx\\
&-\int\rn^\gamma\wn\nabla \wn \wn\frac{2}{2+|\wn|^2}\,dx.
\end{aligned}
\end{equation}
\end{lemma}
\begin{proof}
Let $\beta\in C^{1}(\R)$. By a simple integration by parts we get 
\begin{equation*}
\begin{aligned}
-\mu\int\Delta(\rn \wn)\wn\bep\,dx=&-\mu\int\Delta\rn|\wn|^2\bep\,dx\\
                                              &+\mu\int\Delta\rn\be\,dx\\
                                              &+\mu\int\rn|\nabla \wn|^2\bep\,dx\\
                                             &+\mu\int\rn\left|\nabla\frac{|\wn|^2}{2}\right|^2\bedp\,dx.
\end{aligned}                                             
\end{equation*}
By multiplying \eqref{eq:wqns2} by $\wn\bep$ and by integrating by parts we get 
\begin{equation*}
\begin{aligned}
\frac{d}{dt}&\int\rn\be\,dx+\mu\int\rn|\nabla \wn|^2\bep\\
&+\int\nabla\rn^{\gamma}\wn\bep\,dx+\mu\int\rn\left|\nabla\frac{|\wn|^2}{2}\right|^2\bedp\,dx\\
&-2\la\int\dive(\rn D\wn)\wn\bep\,dx.
\end{aligned}
\end{equation*}
By integrating by parts the last two terms we get 
\begin{equation*}
\begin{aligned}
\frac{d}{dt}&\int\rn\be\,dx+\mu\int\rn|\nabla \wn|^2\bep\,dx\\
&+\mu\int\rn\left|\nabla\frac{|\wn|^2}{2}\right|^2\bedp\,dx+2\la\int\rn |D\wn|^2\bep\,dx\\
=&-2\la\int\rn D\wn \wn\nabla \wn \wn\bedp\,dx+\int\rn^{\gamma}\dive \wn\bep\,dx\\
&+\int\rn^{\gamma}\wn\nabla \wn \wn\bedp\,dx.
\end{aligned}
\end{equation*}
By choosing $\beta(t)=(1+t)\log(1+t)$ we get \eqref{eq:mv}
\end{proof}
\section{Proof of the Theorem \ref{teo:main}}\label{sect:main}
In this Section we are going to prove the main result of our paper. 
The most important part of the proof will be the strong convergence stated in \eqref{eq:main}; the fact that the limit is a weak solution to \eqref{eq:qns1}-\eqref{eq:qns2} in the sense of Definition \ref{def:ws} will then be a consequence of the strong convergence.\\

{\bf Bounds independent on $n$}\\
\\
First of all we resume the a priori estimates we can infer from the previous Section, given the assumptions \eqref{eq:hyidr}, \eqref{eq:hyidu} we have on the initial data. By the energy estimates in Lemma \ref{lem:s1} and \ref{lem:s3bis} we have
\begin{equation}\label{eq:c1}
\begin{aligned}
\|\sqrt{\rho_n}u_n\|_{L^\infty_tL^2_x}\leq C, &\;\|\nabla\sqrt{\rho_n}\|_{L^\infty_tL^2_x}\leq C\\
\|\rho_n\|_{L^\infty_t(L^1_x\cap L^\gamma_x)}\leq C,&\; \|\sqrt{\rho_n}\nabla u_n\|_{L^2_{t, x}}\leq C\\
\|\sqrt{\rho_n}w_n\|_{L^\infty_tL^2_x}\leq C,&\;\|\sqrt{\rho_n}\nabla w_n\|_{L^2_{t, x}}\leq C\\
\|\nabla\rho_n^{\gamma/2}\|_{L^2_{t, x}}\leq C.
\end{aligned}
\end{equation}
In what follows we will also need the following uniform bounds
\begin{equation}\label{eq:h2rho}
\|\sqrt{\rho_n}\|_{L^2_tH^2_x}\leq C
\end{equation}
and 
\begin{equation}\label{eq:pressure}
\|\rho_n^\gamma\|_{L^{5/3}_{t, x}}\leq C.
\end{equation}
The first one comes from the fact that
\begin{equation*}
\|\nabla^2\sqrt{\rho_n}\|_{L^2}\lesssim\|\sqrt{\rho_n}\nabla^2\log\sqrt{\rho_n}\|_{L^2},
\end{equation*}
see for example Lemma 6 in \cite{JM} or Lemma 2.1 in \cite{VY2} for a proof, and from the conservation of total mass. For the second one we use the bound $\|\nabla\rho_n^{\gamma/2}\|_{L^2_{t, x}}\leq C$. The Sobolev embedding then implies that $\{\rho_n^\gamma\}\subset L^1_tL^3_x$ is uniformly bounded. By interpolating this with the bound $\|\rho_n^\gamma\|_{L^\infty_tL^1_x}\leq C$ we then have \eqref{eq:pressure}.
By using those bounds we are now able to prove the Mellet-Vasseur type setimate for $(\rho_n, w_n)$ solutions to \eqref{eq:wqns1}-\eqref{eq:wqns2}.
\begin{remark}
Let us remark that \eqref{eq:h2rho} cannot help us improving the bound \eqref{eq:pressure} on the pressure. This in turn implies that to control the left-hand side of \eqref{eq:mv} we still need the restriction $\gamma<3$ as in \cite{MV}. 
\end{remark}
\begin{lemma}\label{lem:c1}
Let $\rn^0$ and $\un^0$ satisfy \eqref{eq:hyidr}, \eqref{eq:hyidu}. Then, there exists a constant $C>0$ independent on $n$ such that 
\begin{equation}\label{eq:mvn}
\sup_t\int\rn\left(1+\frac{|\wn|^2}{2}\right)\log\left(1+\frac{|\wn|^2}{2}\right)\,dx\leq C.
\end{equation}
\end{lemma}
\begin{proof}
The proof follows the same line of Lemma 4.3 in \cite{MV} and we sketch it just for sake of completeness. For convenience of the reader we write again the identity derived in Lemma \ref{lem:s4}. 
\begin{equation}\label{eq:mvn1}
\begin{aligned}
&\frac{d}{dt}\int\rn\left(1+\frac{|\wn|^2}{2}\right)\log\left(1+\frac{|\wn|^2}{2}\right)\,dx\\
&+\mu\int\rn|\nabla \wn|^2\,dx+\mu\int\rn|\nabla \wn|^2\log\left(1+\frac{|\wn|^2}{2}\right)\,dx\\
&+2\la\int\rn|D\wn|^2\,dx+2\la\int\rn|D\wn|^2\log\left(1+\frac{|\wn|^2}{2}\right)\,dx\\
&+\mu\int\rn\left|\nabla\frac{|\wn|^2}{2}\right|^2\frac{2}{2+|\wn|^2}\,dx\\
&=-2\la\int\rn D\wn \wn\nabla \wn \wn\frac{2}{2+|\wn|^2}\,dx\\
&-\int\rn^\gamma\dive\wn\left(1+\log\left(1+\frac{|\wn|^2}{2}\right)\right)\,dx\\
&-\int\rn^\gamma\wn\nabla \wn \wn\frac{2}{2+|\wn|^2}\,dx.
\end{aligned}
\end{equation}
First, we note that the following estimate holds:
\begin{equation}\label{eq:a}
\begin{aligned}
&-2\la\iint\rn D\wn \wn\nabla \wn \wn\frac{2}{2+|\wn|^2}\,dxdt\\
&-\iint\rn^\gamma\dive\wn\left(1+\log\left(1+\frac{|\wn|^2}{2}\right)\right)\,dxdt\\
&-\iint\rn^\gamma\wn\nabla \wn \wn\frac{2}{2+|\wn|^2}\,dxdt\leq C\iint\rn|\nabla\wn|^2\,dxdt\\
&+C\iint\rn^\gamma|\nabla\wn|\left(1+\log\left(1+\frac{|\wn|^2}{2}\right)\!\!\right)\,dxdt\\
&\leq \frac{C}{2\mu}\iint\rn^{2\gamma-1}\left(1+\log\left(1+\frac{|\wn|^2}{2}\right)\!\!\right)\,dxdt\\
&+C\iint\rn|\nabla\wn|^2\,dxdt+\frac{\mu}{2}\iint\rho|\nabla\wn|^2\left(1+\log\left(1+\frac{|\wn|^2}{2}\right)\!\!\right)\,dxdt,
\end{aligned}
\end{equation}
where the last estimate follows from Cauchy-Schwartz and Young inequality. Now, by integrating in time \eqref{eq:mvn1} and by using the apriori bounds in \eqref{eq:c1} we have
\begin{equation*}
\begin{aligned}
&\sup_t\int\rn\left(1+\frac{|\wn|^2}{2}\right)\log\left(1+\frac{|\wn|^2}{2}\right)\,dx\leq\\
&C+\iint\rn^{2\gamma-1}\left(1+\log\left(1+\frac{|\wn|^2}{2}\right)\right)\,dxdt.
\end{aligned}
\end{equation*}
Next, let $\delta\in(0,2)$. By using H\"older inequality we have 
\begin{equation}\label{eq:mv5}
\begin{aligned}
&\iint\rn^{2\gamma-1}\left(1+\log\left(1+\frac{|\wn|^2}{2}\right)\right)\,dxdt\\
&\leq\int\!\!\left(\int\rn^{(2\gamma-\delta/2-1)(2/(2-\delta))}dx\right)^{\frac{2-\delta}{2}}\!\!\left(\int\rn \left(1+\log\left(1+\frac{|\wn|^2}{2}\right)\right)^{\frac{2}{\delta}}\!dx\right)\!dt\\
&\leq C\int\!\!\left(\int\rn^{(2\gamma-\delta/2-1)(2/(2-\delta))}\,dx\right)^{(2-\delta)/2}\,dt.
\end{aligned}
\end{equation}
where in the last inequality we have used the bounds in \eqref{eq:c1}. Finally, by using \eqref{eq:pressure} we have that for $\delta$ small the integral in the last line of \eqref{eq:mv5} is bounded under the condition $\gamma<3$. Then \eqref{eq:mvn} is proved. 
 
\end{proof}

{\bf Convergence}\\
\\
By using the uniform estimate in Lemma \ref{lem:c1} we are now able to prove the strong convergence of $\{\sqrt{\rho_n}\}$ and $\{\sqrt{\rho_n}u_n\}$, as stated in \eqref{eq:main}.
\begin{lemma}\label{lem:5.1}
Let $\{(\rn, \un)\}_n$ be a sequence of solutions of \eqref{eq:qns1}-\eqref{eq:qns2} and $\wn=\un+\mu\nabla\log\rn$. Then up to subsequences there exist a function $\rrho$ and a vector $m_1$ such that 
\begin{align}
&\rrn\rightarrow\rrho\textrm{ strongly in }L^{2}(0,T;H^{1}(\T)),\label{eq:strong1}\\
&\rn\un\rightarrow m_1\textrm{ strongly in }L^{2}(0,T;L^{p}(\T))\textrm{ with }p\in[1,3/2),\label{eq:strong3}\\
&\rn^{\gamma}\rightarrow\rho^{\gamma}\textrm{ strongly in }L^{1}((0,T)\times\T).\label{eq:strong4}
\end{align}
\end{lemma}
\begin{proof}
Let us consider the continuity equation \eqref{eq:qns1}. By assuming $\rho_n>0$ we may write
\begin{equation}\label{eq:51}
\partial_t\rrn=-\frac{\rrn}{2}\dive\un-\dive(\rrn\un)+\nabla\un\rrn,
\end{equation}
and, by the uniform bounds we have in \eqref{eq:c1}, we have that
\begin{equation*}
\{\partial_t\rrn\}_n\textrm{ is uniformly bounded in }L^{2}(0,T;H^{-1}(\T)).
\end{equation*} 
Then, since $\{\rrn\}_n$ is uniformly bounded in $L^{2}(0,T;H^{2}(\T))$ by using Aubin-Lions Lemma we get \eqref{eq:strong1}. 
To prove \eqref{eq:strong3} we first notice that
\begin{equation}\label{eq:51bis}
\{\nabla(\rn\un)\}_n\textrm{ is uniformly bounded in }L^{2}(0,T;L^{1}(\T)).
\end{equation}
This is again a consequence of the bounds in \eqref{eq:c1}. Let us now consider the equation \eqref{eq:qns2} for the momentum; by using \eqref{eq:quantum} we have
\begin{equation*}
\begin{aligned}
\partial_t(\rn\un)=&-\dive(\rrn\un\otimes\rrn\un)-\nabla p(\rn)-4\kappa^2\dive(\nabla\rrn\otimes\nabla\rrn)\\
&-2\nu\dive(\rrn\un\otimes\nabla\rrn)-2\nu\dive(\nabla\rrn\otimes\rrn\un)\\
&+2\nu\dive(D(\rrn\rrn\un))+\kappa^2\Delta\nabla\rn\\
&=I^{n}_1+I^{n}_2+I^{n}_3+I^{n}_4+I^{n}_5+I^{n}_6+I^{n}_7.
\end{aligned}
\end{equation*}
By using the bounds \eqref{eq:c1} we get for $i=1,...,5$ 
\begin{equation}\label{eq:52}
\{I^{n}_i\}_n\textrm{ is uniformly bounded in }L^{\infty}(0,T;W^{-1,1}(\T)).
\end{equation}                        
Then, by using again \eqref{eq:c1} we have that $\{\diver\symmD(\rho_nu_n)\}\subset L^2(0, T; W^{-2, \frac32})$ is uniformly bounded, and analogously for $\{\Delta\nabla\rho_n\}\subset L^2(0, T; W^{-2, \frac32})$. By the Sobolev embedding $W^{-1, 1}\subset W^{-2, \frac32}$ we then have that
\begin{equation*}
\{\partial_t(\rn\un)\}_n\textrm{ is uniformly bounded in }L^{2}(0,T;W^{-2,\frac{3}{2}}(\T)).
\end{equation*}
Again we apply the Aubin-Lions Lemma and thus \eqref{eq:strong3} is proved.
Finally, the convergence \eqref{eq:strong4} is easily obtained by using \eqref{eq:strong1} and the bound \eqref{eq:pressure}.
\end{proof}
We are now able to prove the strong convergence of $\sqrt{\rho_n}u_n$ in $L^2$, which is the main point in proving Theorem \ref{teo:main}.
\begin{lemma}
Let $(\rn,\un)$ be a sequence of solutions of \eqref{eq:qns1}-\eqref{eq:qns2} and let $\wn=\un+\mu\nabla\log\rho_n$. Then, up to subsequences we have that 
\begin{equation}\label{eq:main2}
\rrn\un\rightarrow\sqrt{\rho}u\textrm{ strongly in }L^{2}((0,T)\times\T),
\end{equation}
where $u$ is defined $m/\rho$ on $\{\rho>0\}$ and $0$ on $\{\rho=0\}$. 
\end{lemma}
\begin{proof}
From Lemma \ref{lem:5.1} we can extract a further subsequence such that 
\begin{equation}\label{eq:convmom}
\begin{aligned}
&\rrn\rightarrow\rrho\textrm{ a.e. in }(0,T)\times\T,\\
&\nabla\rrn\rightarrow\nabla\rrho\textrm{ a.e. in }(0,T)\times\T,\\
&m_{1,n}=\rn\un\rightarrow m_1\textrm{ a.e. in }(0,T)\times\T.
\end{aligned}
\end{equation}
Then, it follows that 
\begin{equation}\label{eq:convmomw}
m_{2,n}:=m_{1,n}+2\mu\rrn\nabla\rrn\rightarrow m_1+2\mu\rrho\nabla\rrho=:m_2,
\end{equation}
a.e. in $(0,T)\times\T$.
Arguing as in \cite{MV} by using \eqref{eq:c1} and Fatou Lemma we have that 
\begin{equation}
 \iint\liminf_n \frac{m_{1,n}^2}{\rn}\,dxdt\leq \liminf_n\iint \frac{m_{1,n}^2}{\rn}\,dxdt< \infty.
\end{equation}
This implies that $m_{1}=0$ a.e. on $\{\rho=0\}$. Let us define the following limit velocity
\begin{equation*}
u=\left\{
\begin{array}{cc}
\displaystyle{\frac{m_1}{\rho}}  & \textrm{ on }\{\rho>0\}  \\
\\
0  & \textrm{ on }\{\rho=0\}.  \\
\end{array}
\right. 
\end{equation*}
In this way we have that $m_1=\rho u$ and $m_1/\sqrt{\rho}\in L^{\infty}(0,T;L^{2}(\T))$. 
Then from \eqref{eq:convmomw} we have that $m_2=m_1+2\mu\rrho\nabla\rrho$ and since $\nabla\rrho$ is finite almost everywhere we also have that $m_2=0$ on the set $\{\rho=0\}$. This in turn implies that after defining the following limit velocity 
\begin{equation*}
w=\left\{
\begin{array}{cc}
\displaystyle{\frac{m_1}{\rho}+2\mu\frac{\rrho\nabla\rrho}{\rho}}  & \textrm{ on }\{\rho\not=0\}  \\
\\
0  & \textrm{ on }\{\rho=0\},  \\
\end{array}
\right. 
\end{equation*}
we have that $m_2=\rho w$ and 
\begin{equation*}
\frac{m_2}{\rrho}=\rrho u+2\mu\nabla\rrho\in L^{\infty}(0,T;L^{2}(\T)).
\end{equation*}
Now we can prove \eqref{eq:main2}. First, by using \eqref{eq:convmom}, Lemma \ref{lem:c1} and Fatou Lemma we get that 
\begin{equation}\label{eq:mvfinal}
\sup_t\int\rho|w|^2\log\left(1+\frac{|w|^2}{2}\right)\,dx\leq\sup_t\int\rn|\wn|^2\log\left(1+\frac{|\wn|^2}{2}\right)\,dx\leq C.
\end{equation}
Then, we note that for any fixed $M>0$
\begin{equation}\label{eq:convtro}
\rrn\wn\chi_{|\wn|\leq M}\rightarrow\rrho w\chi_{|w|\leq M}
\end{equation}
a.e. in $(0, T)\times\Omega$.
Indeed, in $\{\rho\not=0\}$ it holds
\begin{equation}
\rrn\wn=\frac{m_{1,n}}{\rrn}+\nabla\rrn\rightarrow\frac{m_1}{\rrho}+\nabla\rrho\textrm{   a.e.}
\end{equation} 
While, in $\{\rho=0\}$ we have 
\begin{equation}
|\rrn\wn\chi_{|\wn|<M}|\leq M\rrn\rightarrow 0\textrm{   a.e.}
\end{equation}
Then, 
\begin{equation*}
\begin{aligned}
\iint|\rrn\un-\rrho u|^2\,dxdt&\leq \iint|\rrn\wn-\rrho w|^2\,dxdt+4\mu^2\iint|\nabla\rrn-\nabla\rrho|^2\,dxdt\\
                                 &\leq\iint|\rrn\wn\chi_{|\wn|<M}-\rrho w\chi_{|w|<M}|^2\,dxdt\\
                                 &+2\iint|\rrn\wn|^2\chi_{|\wn|>M}\,dxdt+2\iint|\rrn w|^{2}\chi_{|w|>M}\,dxdt\\
                                 &+4\mu^2\iint|\nabla\rrn-\nabla\rrho|^2\,dxdt\\
                                 &\leq\iint|\rrn\wn\chi_{|\wn|<M}-\rrho w\chi_{|w|<M}|^2\,dxdt\\
                                 &+4\mu^2\iint|\nabla\rrn-\nabla\rrho|^2\,dxdt\\
                                 &+\frac{2}{\log(1+M)}\iint\rn|\wn|^2\log\left(1+\frac{|\wn|^{2}}{2}\right)\,dxdt\\
                                 &+\frac{2}{\log(1+M)}\iint\rho|w|^2\log\left(1+\frac{|w|^2}{2}\right)\,dxdt.
\end{aligned}                                 
\end{equation*}
The first term vanishes by using dominated convergence and \eqref{eq:convtro}, the second term converges to $0$ because of \eqref{eq:strong1} and finally the last two term goes to $0$ by using \eqref{eq:mvfinal} and by sending $M\rightarrow \infty$. 
\end{proof}

\section*{Acknowledgement}
We would like to thank Prof. Pierangelo Marcati for some useful comments.

\end{document}